\documentclass[11pt]{amsart}
\usepackage[latin1]{inputenc}
\usepackage[T1]{fontenc}
\usepackage{amssymb,mathrsfs}

\newcounter{c}
\newtheorem{Theo}{Theorem}
\newtheorem{Prop}[c]{Proposition}
\newtheorem{Lemma}[c]{Lemma}
\newtheorem{Cor}[c]{Corollary}
\theoremstyle{definition}
\newtheorem{Def}[c]{Definition}
\theoremstyle{remark}
\newtheorem{Ex}[c]{Example}
\newtheorem*{Remark}{Remark}

\thanks{The first author was supported by CNPq (Brazil) and the second author was supported by Fapesp, grant no. 2013/22.802-1.}

\author[E. Hitomi]{Eduardo Eizo Aramaki Hitomi}
\address{Department of Mathematics, State University of Campinas, Campinas, SP, 13083-859, Brazil.}
\email{ehitomi@ime.unicamp.br}

\author[F. Yasumura]{Felipe Yukihide Yasumura}
\address{Department of Mathematics, State University of Campinas, Campinas, SP, 13083-859, Brazil.}
\email{ra091138@ime.unicamp.br}


\title{On the combinatorics of commutators of Lie algebras}

\begin{document}
\begin{abstract}
	Motivated by the combinatorial properties of products in Lie algebras, we investigate the subset of permutations that naturally appears when we write the long commutator $[x_1,x_2,...,x_m]$ as a sum of associative monomials. We characterize this subset and find some useful equivalences. Moreover, we explore properties concerning the action of this subset on sequences of m elements. In particular we describe sequences that share some special symmetries which can be useful in the study of combinatorial properties in graded Lie algebras.
\end{abstract}
\maketitle

\section{Introduction.} For $m\in\mathbb{N}$, let $I_m=\{1,2,\ldots,m\}$ and $\mathcal{S}_m$ denote the set of bijections of $I_m$. We assume that the elements of $\mathcal{S}_m$ acts by left on $I_m$, that is, if $\sigma,\tau\in \mathcal{S}_m$ then $\sigma\circ\tau$ stands for applying $\tau$ first and then $\sigma$.

	We are interested in studying the subset $\mathscr{T}_m$ of $\mathcal{S}_m$, where given an associative algebra $A$ and given elements $x_1,x_2,\ldots,x_m$,
	\[
		[x_1,x_2,\ldots,x_m]=\sum_{\sigma\in\mathscr{T}_m}\pm x_{\sigma(1)}x_{\sigma(2)}\cdots x_{\sigma(m)},
	\]
	where the long commutator is left normed, that is, we define
	\begin{eqnarray*}
		[x_1,x_2]&=&x_1x_2-x_2x_1,\\{}
		[x_1,x_2,\ldots,x_m]&=&[[x_1,x_2,\ldots,x_{m-1}],x_m],\quad\text{ for $m>2$}.
	\end{eqnarray*}
	The set $\mathscr{T}_m$, its properties and relations with Lie algebras was intensely investigated (see, for instance, \cite{BL1992,G}).

	We state some equivalences of $\mathscr{T}_m$ (see Corollary \ref{prod_equiv} below), and then we will derive properties concerning actions by permutations of sequences restricted to $\mathscr{T}_m$. This theory is an important tool to study group gradings on the algebra of upper triangular matrices algebras viewed as a Lie algebra. We hope the theory developed here can be useful to other algebras.

\section{Some equivalences.} We denote by $\mathcal{S}_m$ the symmetric group permuting the symbols 1, 2, \dots, $m$. Following \cite{BL1992}, let
	\[
		\mathscr{T}_m=\left\{\sigma\in \mathcal{S}_m\mid\sigma(1)>\ldots>\sigma(t)=1,\sigma(t+1)<\ldots<\sigma(m),t=1,\ldots,m\right\}.
	\]
	There are some equivalences to define the above set:
	\begin{Lemma}[\cite{BL1992}]\label{first_lemma}
		Let $\sigma\in \mathcal{S}_m$. The following conditions are equivalent:
		\begin{enumerate}
			\item $\sigma\in\mathscr{T}_m$;
			\item there exists $r$ such that: $\sigma(j)>\sigma(j+1)$ if and only if $1\le j\le r$;
			\item there exists $j_1>j_2>\cdots>j_r>1$ such that
				\[
					\sigma=(j_r \dots 1)\cdots(j_1 \dots 1).
				\]
				Moreover, $j_i=\sigma(i)$ for $i=1,2,\ldots,r$.
		\end{enumerate}
		Also given an associative algebra $A$ and $x_1,x_2,\ldots,x_m\in A$,
		\[
			[x_1,x_2,\ldots,x_m]=\sum_{\sigma\in\mathscr{T}_m}(-1)^{\sigma^{-1}(1)-1}x_{\sigma(1)}x_{\sigma(2)}\cdots x_{\sigma(m)}.
		\]
	\end{Lemma}
	The last assertion of the previous lemma says that the set $\mathscr{T}_m$ is indeed the set we want to study.

	If we write the elements of $\mathcal{S}_m$ in two row notation, like
	\[
		\sigma=\left(\begin{array}{cccc}1&2&\ldots&m\\\sigma(1)&\sigma(2)&\ldots&\sigma(m)\end{array}\right),
	\]
	then we can easily recognize if $\sigma$ is an element of $\mathscr{T}_m$ or not, by definition. Also, it is easy to see that for every $r=1,2,\ldots,m$, the numbers $1,2,\ldots,r$ in the second row appear together, in ``only one block".

We draw the reader's attention that in general $\mathscr{T}_m$ is not even a subsemigroup of $\mathcal{S}_m$.
\begin{Ex} Consider the  permutations
			\[
				\sigma_1=\left(\begin{array}{ccc}1&2&3\\2&1&3\end{array}\right),\quad\sigma_2=\left(\begin{array}{ccc}1&2&3\\3&1&2\end{array}\right),\quad\sigma_2\circ\sigma_1=\left(\begin{array}{ccc}1&2&3\\1&3&2\end{array}\right).
			\]
			Then, using remark above, $\sigma_1$, $\sigma_2\in\mathscr{T}_3$ but $\sigma_2\circ\sigma_1\notin\mathscr{T}_3$. Hence $\mathscr{T}_m$ is not in general a subgroup of $\mathcal{S}_m$.
\end{Ex}

	It is not hard to derive the following equivalence:
	\begin{Lemma}\label{originalequivalence}
		Let $\sigma\in \mathcal{S}_m$. Then $\sigma\in\mathscr{T}_m$ if and only if $\sigma$ satisfies the following condition: there exists an integer $t$, with $1\le t\le m$, such that
	\begin{enumerate}
		\renewcommand{\labelenumi}{(\roman{enumi})}
		\item $\sigma(t)=1$;
		\item for every positive integers $k_1$, $k_2\ge0$ such that $k_1+k_2<m$ and
			\[
				\{\sigma(t-k_1),\sigma(t-k_1+1),\ldots,\sigma(t+k_2)\}=\{1,2,\cdots,k_1+k_2+1\},
			\]
			it holds that either
			\begin{itemize}
				\item $t-k_1-1\ge 1$ and $\sigma(t-k_1-1)=k_1+k_2+2$ or
				\item $t+k_2+1\le m$ and $\sigma(t+k_2+1)=k_1+k_2+2$.
			\end{itemize}
	\end{enumerate}
	We denote $\mathscr{T}_m^{(t)}=\{\sigma\in \mathscr{T}_m\mid \sigma(t)=1\}$.
\end{Lemma}
The previous lemma is useful for inductive arguments concerning the elements of $\mathscr{T}_m$.

The following is convenient for our applications:
\begin{Lemma}\label{lemma_dois}
	Let $r_1$, $r_2$, \dots, $r_m$ be strictly upper triangular matrix units such that their associative product $r_1r_2\cdots r_m\ne0$. Then
	\begin{enumerate}
		\renewcommand{\labelenumi}{(\roman{enumi})}
		\item $r_{\sigma^{-1}(1)}r_{\sigma^{-1}(2)}\cdots r_{\sigma^{-1}(m)}\ne0$ if and only if $\sigma=1$;
		\item $[r_{\sigma^{-1}(1)},r_{\sigma^{-1}(2)},\cdots,r_{\sigma^{-1}(m)}]\ne0$ if and only if $\sigma\in\mathscr{T}_m$.
	\end{enumerate}
\end{Lemma}
\begin{proof}
	Let $r_1$, $r_2$, \dots, $r_m$ be strictly upper triangular matrices such that their associative product $r_1r_2\cdots r_m\ne0$. 
			Then $r_{\sigma(1)}r_{\sigma(2)}\cdots r_{\sigma(m)}\ne0$ if and only if $\sigma=1$, since every $r_l=e_{i_lj_l}$ and $i_l<j_l$. Note that $r_lr_k\ne0$ if and only if $j_l=i_k$.

			Now consider the Lie commutator. Let $\sigma\in \mathcal{S}_m$ be such that  
			\begin{equation}
				\label{tlie}
				[r_{\sigma^{-1}(1)},r_{\sigma^{-1}(2)},\ldots,r_{\sigma^{-1}(m)}]\ne0
				\end{equation}
			Assume $t=\sigma^{-1}(1)$. Since $[r_{\sigma^{-1}(1)},r_{\sigma^{-1}(2)}]\ne0$, it follows that either  $\sigma^{-1}(2)=t+1$ or $\sigma^{-1}(2)=t-1$. Iterating this and using induction we obtain that $\sigma\in\mathscr{T}_m$, by Lemma \ref{originalequivalence}. The same idea can be used to prove the converse, that is (\ref{tlie}) holds for each $\sigma\in\mathscr{T}_m$.
\end{proof}

We remark that if $m_1\le m_2$ we can consider $\mathcal{S}_{m_1}$ as a subgroup of $\mathcal{S}_{m_2}$ in the usual manner, that is the elements of $\mathcal{S}_{m_1}$ fix all symbols $t>m_1$. Using the same identification, we can consider $\mathscr{T}_{m_1}$ as a subset of $\mathscr{T}_{m_2}$. From definition we obtain an interesting consequence. 
\begin{Cor}\label{cor_cite}
	Let $\sigma\in\mathscr{T}_m$ and $m_1=\sigma(1)$. Then $\sigma\in\mathscr{T}_{m_1}$ (that is, $\sigma(t)=t$ for $t>m_1$).
\end{Cor}
\begin{proof}
	Direct from Lemma \ref{first_lemma}, item 3.
\end{proof}

Now, consider the following special subset of permutations in $\mathscr{T}_m$:
\begin{Def}
	For every $i=1,2,\dots,m$, the $i$-reverse permutation is given by
	\[
		\tau_i=\left(\begin{array}{ccccccccc}
			1&2&\ldots&i-1&i&i+1&i+2&\ldots&m\\
			i&i-1&\ldots&2&1&i+1&i+2&\ldots&m
			\end{array}\right).
	\]
\end{Def}
We remark that, for every $i=1$, 2, \dots, $m$ one has $\tau_i\in\mathscr{T}_m^{(i)}$ and $\tau_i^2=1$. Moreover, $\tau_{i-1}\circ\tau_i=(i\quad i-1\quad\ldots\quad1)$. So, by Lemma \ref{first_lemma}.(3), we obtain:

\begin{Cor}\label{prod_equiv}
	$\displaystyle\mathscr{T}_m=\left\{\tau_1^{i_1}\circ\tau_2^{i_2}\circ\cdots\circ\tau_m^{i_m}\mid i_1,i_2,\dots,i_m\in\{0,1\}\right\}$.
\end{Cor}

Note that it is easy to obtain the decomposition of elements of $\mathscr{T}_m$ into product of $\tau_i$. Corollary \ref{prod_equiv} also tells when the product of two elements of $\mathscr{T}_m$ still belongs to $\mathscr{T}_m$.

\begin{Remark}Now, using ideas presented in \cite{BL1992} concerning elements in the group algebra, and using Corollary \ref{prod_equiv}, we obtain the following extra fact. Consider the element $v_m=\sum_{\sigma\in\mathscr{T}_m}(-1)^{\sigma^{-1}(1)-1}\sigma$ in the group algebra $\mathbb{Z}\mathcal{S}_m$. Then:
	\begin{enumerate}
		\item\cite{BL1992} we have
			\begin{align*}
				\displaystyle v_m&=\prod_{i=m}^2(1-(i,i-1,\ldots,1))\\
						&=(1-(m,m-1,\ldots,1))\cdots(1-(2\quad1)).
			\end{align*}
		\item  $\displaystyle v_m=\prod_{i=2}^m(1-(-1)^i\tau_i)=(1-\tau_2)(1+\tau_3)\cdots(1-(-1)^m\tau_m)$.
	\end{enumerate}
\end{Remark}

\section{Action on sequences.} Let $X$ be any set, and fix a sequence $s\in X^m$. Then we have a left action of $\mathcal{S}_m$ on the elements $s=(s_1,s_2,\dots,s_m)\in X^m$ permuting the order of the elements
	\[
		\sigma (s_1,s_2,\dots,s_m)=(s_{\sigma^{-1}(1)},s_{\sigma^{-1}(2)},\dots,s_{\sigma^{-1}(m)}),\quad\sigma\in \mathcal{S}_m.
	\]

We are interested in the following notion.
\begin{Def}
	Given two sequences $s, s'\in X^m$, we say that $s$ and $s'$ are mirrored if $\mathscr{T}_ms=\mathscr{T}_ms'$.
\end{Def}
Equivalently, $s$ and $s'$ are mirrored if and only if for all $\tau', \sigma\in\mathscr{T}_m$, we can find $\tau, \sigma'\in\mathscr{T}_m$ such that $\sigma s=\sigma's'$ and $\tau' s'=\tau s$.

We give two examples of mirrored sequences.

\noindent\textbf{Notation.} Given $s\in X^m$, we denote by $\text{rev}\,s:=\tau_m s$, the reverse sequence of $s$.
\begin{Ex}
	The trivial example: for any $s\in X^m$, $s$ and $s$ are mirrored.
\end{Ex}
\begin{Ex}
	If $s,s'\in X^m$ are mirrored then $s$ and $\text{rev}\,s'$ are also mirrored.
\end{Ex}
\begin{proof}
	Let $\sigma\in\mathscr{T}_m$. Then there exists $\sigma'\in\mathscr{T}_m$ such that $\sigma s=\sigma's'$, since $s$ and $s'$ are mirrored. Also $\sigma'':=\sigma'\tau_m\in\mathscr{T}_m$, by Corollary \ref{prod_equiv}, and \[\sigma''\text{rev}\,s'=\sigma'\tau_m\tau_m s'=\sigma' s'=\sigma s.\]

	Conversely, given $\tau'\in\mathscr{T}_m$, we have $\tau'\tau_m\in\mathscr{T}_m$ and there exists $\tau\in\mathscr{T}_m$ such that $\tau s=\tau'\tau_m s'=\tau'\text{rev}\,s'$. Hence $s$ and $\text{rev}\,s'$ are mirrored.
\end{proof}

We will prove that these two examples are the unique way to produce mirrored sequences. A precise statement of our main result is as follows.
\begin{Theo}\label{mainthm}
	Let $X$ be any set, $m\in\mathbb{N}$, and let $s,s'\in X^m$. Then $\mathscr{T}_ms=\mathscr{T}_ms'$ if and only if $s=s'$ or $s=\text{rev}\,s'$.
\end{Theo}
An equivalent statement is the following:
\begin{Cor}
	Let $X$ be any set and $s,s'\in X^m$. Then $s=s'$ or $s=\text{rev}\,s'$ if and only if for any $\sigma,\tau'\in\mathscr{T}_m$, we can find $\sigma',\tau\in\mathscr{T}_m$ such that $\sigma s=\sigma's'$ and $\tau s=\tau's'$.
\end{Cor}

We have already proved one of the implications of Theorem \ref{mainthm} in the two previous examples. The last section is dedicated exclusively to prove the converse.

\section{Proof of Theorem \ref{mainthm}.}
We fix some notations.

\noindent\textbf{Notation.} We denote $I_m=\{1,2,\ldots,m\}$. Note that any $s\in X^m$ can be viewed as a function $s:I_m\to X$ and for any $\sigma\in \mathcal{S}_m$, $\sigma s=s\circ\sigma^{-1}:I_m\to X$ (equality of functions). Given integers $0< m_2\le m$, we denote by $I_{-m_2}^{(m)}=\{m,m-1,\ldots,m-m_2+1\}$, and $I_{0}^{(m)}=\emptyset$ (the last $m_2$ elements of $I_m$).

First we deduce several useful properties.

\begin{Def}
	Let $s$, $s':I_m\to X$, and set $A=s(1)$. A coincidence of $(s,s')$ is a pair $(m_1,m_2)$ where $m_1>0$ and $m_2\ge0$ are integers satisfying:
	\begin{enumerate}
		\item $s(i)=s'(i)=A$, for all $i=1,2,\ldots,m_1,$ and $i=m,m-1,\ldots,m-m_2+1$,
		\item $s(m_1+1)\ne A$, $s'(m_1+1)\ne A$, $s(m-m_2)\ne A$ and $s'(m-m_2)\ne A$.
	\end{enumerate}
	In this context, we denote $m'=m_1+m_2$, $I_m^{(-)}=I_{m_1}\cup I_{-m_2}^{(m)}$ and $I_m'=I_m\setminus I_m^{(-)}$.
\end{Def}

\begin{Lemma}
	Let $m_1,m_2\ge0$ with $m_1+m_2\le m$. Then $\sigma\in\mathscr{T}_m$ satisfies:
	\begin{itemize}
		\item $\sigma(m-i+1)=m-i+1$, for $i=1,2,\ldots,m_2$,
		\item $\sigma(i)=m-m_2-i+1$, for $i=1,2,\ldots,m_1$,
	\end{itemize}
	if and only if $\sigma\in\mathscr{T}_{m-m_1-m_2}\circ\tau_{m-m_2}$.
\end{Lemma}
\begin{proof}
Let $\sigma\in\mathscr{T}_{m-m_1-m_2}\circ\tau_{m-m_2}$, then clearly $\sigma(m-i+1)=m-i+1$, for $i=1,2,\ldots,m_2$. Also, writing $\sigma=\sigma'\circ\tau_{m-m_2}$, we have $\sigma(1)=\sigma'(m-m_2)=m-m_2$, since $\sigma'\in\mathscr{T}_{m-m_1-m_2}$, and using the same idea, we see that $\sigma$ satisfies the second condition.

We can prove the converse using the same idea.
\end{proof}
\begin{Ex}
	It is easy to represent such a permutation $\sigma\in\mathscr{T}_{m-m_1-m_2}\circ\tau_{m-m_2}$: the last $m_2$ entries of $\sigma$ are:
	\[
		\sigma=\left(\begin{array}{ccccc}
			\ldots&m-m_2+1&\ldots&m-1&m\\
			\ldots&m-m_2+1&\ldots&m-1&m
			\end{array}\right)
	\]
	and the first $m_1$ entries are:
	\[
		\sigma=\left(\begin{array}{ccccc}
			1&2&\ldots&m_1&\ldots\\
			m_0&m_0-1&\ldots&m_0-m_1+'&\ldots
			\end{array}\right),
	\]
	where $m_0=m-m_2$.
\end{Ex}
The following example serves as a motivation for considering the above permutations and  the notion of coincidence:
\begin{Ex}\label{cool_example}
	Let $(m_1,m_2)$ be a coincidence for $(s,s')$, and let $A=s(1)$. Then $s$ and $s'$ are written as:
	\begin{eqnarray*}
		s&=&(\underbrace{A,\ldots,A}_{\text{$m_1$ times}},B_1,\ldots,B_l,\underbrace{A,\ldots,A}_{\text{$m_2$ times}})\\
		s'&=&(A,\ldots,A,B_1',\ldots,B_l',A,\ldots,A),
	\end{eqnarray*}
	where $B_1\ne A$, $B_1'\ne A$, $B_l\ne A$ and $B_l'\ne A$. A permutation $\sigma\in\mathscr{T}_{m-m'}\circ\tau_{m-m_2}$ will act like:
	\[
		\sigma s=(C_1,\ldots,C_l,\underbrace{A,\ldots,A}_{\text{$m'$ times}}),
	\]
	where $C_l\ne A$.
\end{Ex}

Let $s:I_m\to X$. Then there is a coincidence $(m_1,m_2)$ of $(s,s)$, and we can consider the restriction $s_0=s\mid_{I_m'}$. Denote $\varphi:n\in I_{m-m'}\mapsto n+m_1\in I_m'$, and let
\[
	\mathscr{T}_m'=\left\{\varphi\circ\sigma\circ\varphi^{-1}\mid\sigma\in\mathscr{T}_{m-m'}\right\}.
\]
For every $\sigma'\in\mathscr{T}_m$, we can define an element $\sigma\in\mathscr{T}_{m-m'}\circ\tau_{m-m_2}$ by
\[
	\sigma(i)=\left\{\begin{array}{l}
			\varphi^{-1}\circ\sigma(i),\quad\text{if $i\in I_m'$},\\
			m-i+1-m_2,\quad\text{if $i\in I_{m_1}$},\\
			i,\quad\text{if $i\in I_{-m_2}^{(m)}$}.
		\end{array}\right.
\]
Hence we can see $\mathscr{T}_m'$ as a subset of $\mathscr{T}_{m-m'}\circ\tau_{m-m_2}$. In particular, both sets have the same cardinality, so we have equality of sets. It is also easy to see what an action of $\sigma\in\mathscr{T}_{m-m'}\circ\tau_{m-m_2}$ on $s_0$ will be.
\begin{Ex}
	Given $s:I_m\to X$, write
	\[
		s=(\underbrace{A,\ldots,A}_{\text{$m_1$ times}},B_1,\ldots,B_l,\underbrace{A,\ldots,A}_{\text{$m_2$ times}}),
	\]
	with $B_l\ne A$. An element $\sigma\in\mathscr{T}_{m-m'}\circ\tau_{m-m_2}$ acts like
	\[
		\sigma s=(C_1,\ldots,C_l,A,\ldots,A),
	\]
	with $C_l\ne A$. Using this notation, we will have a coincidence $(m_1,m_2)$ of $(s,s)$, $s_0=s\mid_{I_m'}=(B_1,\ldots,B_l)$ and $\sigma s_0=(C_1,\ldots,C_l)$.
\end{Ex}
\begin{Lemma}
	Using the above notation, if $\sigma\in\mathscr{T}_m$ satisfies
	\begin{align*}
		\{\sigma^{-1}(m),\sigma^{-1}(m-1),\ldots,\sigma^{-1}(m-m'+1)\}=\\
		\{1,2,\ldots,m_1,m,m-1,\ldots,m-m_2+1\},
	\end{align*}
	there exists $\tau\in\mathscr{T}_{m-m'}\circ\tau_{m-m_2}$ such that $\sigma s=\tau s$.
\end{Lemma}
\begin{proof}
	Note that, in this case, $\varphi\circ\sigma(I_m')\subset I_m'$, so we can see $\sigma$ as an element of $\mathscr{T}_m'$, and therefore we can construct one such $\tau\in\mathscr{T}_{m-m'}\circ\tau_{m-m_2}$.
\end{proof}

\begin{Lemma}\label{nice_lemma}
	Let $s,s':I_m\to X$ be mirrored and assume that there is a coincidence $(m_1,m_2)$ of $(s,s')$. Then $s_0=s\mid_{I_m'}$ and $s_0'=s'\mid_{I_m'}$ are mirrored.
\end{Lemma}
\begin{proof}
	For every $\sigma\in\mathscr{T}_{m-m'}\circ\tau_{m-m_2}$, we can find $\tau\in\mathscr{T}_m$ such that $\sigma s=\tau s'$. We consider $\sigma\in\mathscr{T}_m'$, hence it is sufficient to prove that $\tau$ is an element of $\mathscr{T}_m'$. Using the previous lemma, assume that there is $i$ such that $i\in\{m,m-1,\ldots,m-m'+1\}$ and $\tau^{-1}(i)=m_1+1$ (or $\tau^{-1}(i)=m-m_2$). By the definition of coincidence, we know that $s'(m_1+1)\ne A$ and $s'(m-m_2)\ne A$, where $A=s(1)$. Also, by the choice of $\sigma$, we know that $\sigma^{-1}(i)\in\{1,2,\ldots,m_1,m,m-1,\ldots,m-m_2+1\}$. Therefore $\sigma s(i)=A$ (see example \ref{cool_example}). Thus $\tau s'(i)\ne A=\sigma s(i)$, a  contradiction. This proves that we can find $\tau'\in\mathscr{T}_m'$ with $\tau'r'=\tau r$, and in particular, $\sigma s_0=\tau' s_0'$. Since a coincidence $(s,s')$ is also a coincidence of $(s',s)$, we can repeat the argument and prove that $s_0$ and $s_0'$ are mirrored.
\end{proof}

\begin{Def}
	Let $s:I_m\to X$ and $w:I_d\to X$ with $d\le m$. Let
	\[
		\mathscr{O}(s,w)=\left\{(\sigma,i)\in\mathscr{T}_m\times I_m\mid\sigma s(i+j)=w(1+j),\forall j=0,1,2,\ldots,d-1\right\},
	\]
	denote $n:(\sigma,i)\in\mathscr{T}_m\times I_m\mapsto i\in I_m$, and define
	\[
		o_w(s)=\left\{\begin{array}{l}
				\text{min}\{n(x)\mid x\in\mathscr{O}(s,w)\},\text{ if $\mathscr{O}(s,w)\ne\emptyset$},\\
				\infty,\text{ otherwise}.
			\end{array}\right.
	\]
\end{Def}
\begin{Lemma}
	Let $s$, $s':I_m\to X$. If there exists $w:I_d\to X$ such that $o_w(s)\ne o_w(s')$, then $s$ and $s'$ cannot be mirrored.
\end{Lemma}
\begin{proof}
	Assume $i=o_w(s)<o_w(s')$. In this case, there will exist $\sigma\in\mathscr{T}_m$ such that $\sigma s$ has $i$ entries, followed by the entries of $w$. However it is impossible to get such $\sigma'\in\mathscr{T}_m$ satisfying the same property, since in this case we would obtain $o_w(s')\le i$, and in particular, $\sigma' s'\ne\sigma s$, for all $\sigma'\in\mathscr{T}_m$.
\end{proof}
\begin{Ex}
	Let $s=(A,A,B,C,D)$ and $s'=(A,B,C,D,A)$. Then if $w=(A,A)$, we have $o_w(s)=1$ and $o_w(s')=4$, hence $s$ and $s'$ can not be mirrored. This can also be viewed directly: let $\sigma=id\in\mathscr{T}_m$, then there is no $\sigma'\in\mathscr{T}_m$ such that $\sigma' s'=\sigma s=s$.
\end{Ex}

\noindent\textbf{Notation.}\ 
\begin{enumerate}
	\renewcommand{\labelenumi}{(\roman{enumi})}
	\item Given $w_1:I_{d_1}\to X$ and $w_2:I_{d_2}\to X$, we denote by $w=(w_1,w_2)$ the sequence $w:I_{d_1+d_2}\to X$ defined by
		\[
			w(i)=\left\{\begin{array}{l}
				w_1(i),\text{ if $1\le i\le d_1$},\\
				w_2(i-d_1),\text{ if $i>d_1$}.
			\end{array}\right.
		\]
	\item Analogously we define $(w_1,w_2,\ldots,w_p)$.
	\item Given $A\in X$ and $d\in\mathbb{N}$, we denote by $A_d:I_d\to X$ the constant sequence $A_d(1)=\cdots=A_d(d)=A$.
\end{enumerate}

We focus now on the special case where $X=\{A,B\}$ has exactly two symbols.
\begin{Def}
	Let $s:I_m\to X=\{A,B\}$. Let $n_1\ge 0$ be the largest integer such that
	\[
		s(1)=s(2)=\cdots=s(n_1)=A,
	\]
	and, for this $n_1$, let $n_2>0$ be the largest integer such that
	\[
		s(n_1+1)=s(n_1+2)=\cdots=s(n_1+n_2)=B.
	\]
	Continuing this process, we obtain the sequence $\Sigma(s)=(n_1,n_2,\ldots,n_{2t-1},n_{2t})$ where we can have $n_1=0$ and we can have $n_{2t}=0$. We call it the spectrum sequence of $s$. 
\end{Def}
\begin{Def}
	Let $s:I_m\to X=\{A,B\}$. For $i\in\{1,2,\ldots,2t\}$ and $j\in\mathbb{N}\cup\{0\}$, let $\Sigma(s)(l)=0$, for $l\notin\{1,2,\ldots,2t\}$, and let
	\[
		e_i^{(j)}(s)=\left\{\begin{array}{l}
			\Sigma(s)(i),\text{ if $j=0$},\\
			\Sigma(s)(i+j)+\Sigma(s)(i-j),\text{ if $j>0$}.
		\end{array}\right.
	\]

	Furthermore we define $m_A^{(1)}(s)=\max\{\Sigma(s)(2i+1)\mid i=0,1,2,\ldots\}$, and  $I_A^{(1)}(s)=\{2i+1\mid\Sigma(s)(2i+1)=m_A^{(1)}(s)\}$, and inductively for $i>1$,
	\begin{align*}
		&m_A^{(i)}(s)=\left\{\begin{array}{l}\max\{e_l^{(i-1)}(s)\mid l\in I_A^{(i-1)}(s)\},\text{ if $I_A^{(i-1)}(s)\ne\emptyset$},\\0,\text{ otherwise}\end{array}\right.,\\
		&I_A^{(i)}(s)=\{l\in I_A^{(i-1)}(s)\mid e_l^{(i-1)}(s)=m_A^{(i)}(s)\text{ and }m_A^{(i)}(s)>0\}.
	\end{align*}
	We make analogous constructions for the symbol $B$.
\end{Def}

\begin{Lemma}
	Given $s$, $s':I_m\to X=\{A,B\}$, if there exists $i$ such that $m_A^{(i)}(s)\ne m_A^{(i)}(s')$ then $s$ and $s'$ cannot be mirrored.
\end{Lemma}
\begin{proof}
	Assume that $i\in\mathbb{N}$ is such that $m_A^{(j)}(s)=m_A^{(j)}(s')$, for $j<i$, and $m_A^{(i)}(s)>m_A^{(i)}(s')$. Let
	\[
		w=(A_{n_1},B_{n_2},\ldots,C_{(n_i)}),
	\]
	where $n_j=m_A^{(j)}(s)$, for all $j=1,2,\ldots,i$, $C=A$ if $i$ is odd and $C=B$ if $i$ is even. Then $o_w(s)=1\ne\infty=o_w(s')$, hence $s$ and $s'$ are not mirrored.
\end{proof}
An intuitive approach for the numbers $m_A^{(i)}(s)$ is the following. The greatest sequence of A's is $m_A^{(1)}(s)$. Along these maximum number $m_A^{(1)}(s)$, we see the number of B's after and before the sequence of A's, and the greatest number of B's is denoted by $m_A^{(2)}(s)$. We continue inductively.

\begin{Ex}
	Let $s=(A_3,B_3,A_3,B_1)$ and $s'=(A_3,B_1,A_3,B_3)$. Then $m_A^{(1)}(s)=m_A^{(1)}(s')=3$, $m_A^{(2)}(s)=m_A^{(2)}(s')=4$, $m_A^{(3)}(s)=m_A^{(3)}(s')=3$, $m_A^{(j)}(s)=m_A^{(j)}(s')=0,\forall j>3$. But we have
	\[
		m_B^{(2)}(s)=6\ne3=m_B^{(2)}(s'),
	\]
	hence $s$ and $s'$ are not mirrored.
\end{Ex}

Let us consider again $s:I_m\to X=\{A,B\}$, and the spectrum sequence $\Sigma(s)=(n_1,n_2,\ldots,n_{2t-1},n_{2t})$. Note that
\[
	I_A^{(1)}(s)\supset I_A^{(2)}(s)\supset\cdots,
\]
and that there is $n\in\mathbb{N}$ such that $I_A^{(n)}(s)\ne\emptyset$ and $I_A^{(j)
}(s)=\emptyset$ for every $j>n$. In this case, if $i\in I_A^{(n)}(s)$ then this entry satisfies the following condition: $e_i^{(n-1)}$ is either the first non-zero entry of $\Sigma(s)$, or the last non-zero entry of $\Sigma(s)$ or either the sum of the first and the last non-zero entries of $\Sigma(s)$. Note that the last possibility happens if and only if $A$ appears ``in the middle" if we consider the spectrum sequence $\Sigma(s)$.

Now write $\Sigma(s)=(n_1,n_2,\ldots,n_{2t-1},n_{2t})$ and assume $n_1>0$ (we can do it renaming $A$ and $B$, if necessary). Also, define (the ``last entry")
\[
	n_l=\left\{\begin{array}{l}
		n_{2t-1},\text{ if $n_{2t}=0$},\\
		0,\text{ otherwise}.
	\end{array}\right.
\]
We use the analogous notation for a given $s':I_m\to X=\{A,B\}$. Assume that $s$ and $s'$ are mirrored. Then $s'=\sigma s$ for some $\sigma\in\mathscr{T}_m$, and $s=\sigma's'$ for some $\sigma'\in\mathscr{T}_m$. Since $\sigma^{-1}(m)$, $\sigma^{\prime-1}\in\{1,m\}$, we can assume that (changing $s'$ with $\text{rev}\,s'$ if necessary) $n_1'>0$.

Since $s$ and $s'$ are mirrored, we can look at the permutations in $\mathscr{T}_{m-n_1-n_l}\circ\tau_{m-n_l}$ and conclude that necessarily $n_1+n_l=n_1'+n_l'$. Moreover, since $m_C^{(i)}(s)=m_C^{(i)}(s')$ for all $i$ and for all $C\in\{A,B\}$, we obtain necessarily (by the observation above) $n_i=n_j'$, for at least one pair of $i$, $j\in\{1,l\}$. As a consequence, using these two equations, we obtain the following
\begin{Lemma}
	If $s$, $s':I_m\to X=\{A,B\}$ are mirrored then there is a coincidence in $(s,s')$ or in $(s,\text{rev}\,s')$.
\end{Lemma}

Now we use induction in order to prove the following
\begin{Prop}\label{ME_twoelements}
	Let $s$, $s':I_m\to X=\{A,B\}$ be mirrored. Then $s=s'$ or $s=\text{rev}\,s'$.
\end{Prop}
\begin{proof}
	We can assume that $(m_1,m_2)$ is a coincidence for $(s,s')$, by the previous lemma, changing $s'$ to $\text{rev}\,s'$, if necessary. Moreover we can consider $s_0=s |_{I_m'}$ and $s_0'=|_{I_m'}$, which are mirrored, by Lemma \ref{nice_lemma}. By the induction hypothesis, we obtain $s_0=s_0'$ or $s_0=\text{rev}\,s_0'$. The last one implies the following fact: if $\Sigma(s)=(n_1,n_2,\ldots,n_{2t-1},n_{2t})$ then $\Sigma(s')=(n_1,n_{2t-1},\ldots,n_2,n_{2t})$. Since $m_C^{(i)}(s)=m_C^{(i)}(s')$ for all $i$ and for all $C$, we obtain necessarily $s=s'$ or $s=\text{rev}\,s'$.
\end{proof}

Below we consider when the general case can be reduced to that of 2 symbols.
\begin{Def}
	Let $s:I_m\to X$ where $X$ is any set, let $X_0\subset X$ and let $R$ be any symbol (it can be an element of $X$ or not). We define the function
	\[
		\pi_{X_0,R}(s):I_m\to X\cup\{R\}
	\]
	by
	\[
		\pi_{X_0,R}(s)(i)=\left\{\begin{array}{l}
			s(i),\text{ if $s(i)\in X_0$},\\
			R,\text{ otherwise}.
		\end{array}\right.
	\]
\end{Def}
Note that for any $\sigma\in \mathcal{S}_m$, $\pi_{X_0,R}(\sigma s)=\sigma\pi_{X_0,R}(s)$. In particular, if $s$ and $s'$ are mirrored, then so are $\pi_{X_0,R}(s)$ and $\pi_{X_0,R}(s')$, for any choice of $X_0\subset X$ and $R$.

\begin{Ex}
	Let $s=(A,B,C,D,C)$ and $s'=(A,D,C,B,C)$. It is easy to see that $s\ne s'$ and $s\ne\text{rev}\,s'$, and $s$ and $s'$ are not mirrored. On the other hand for every choice of $x, y\in\{A,B,C,D\}$ we have $\pi(s)=\pi(s')$ or $\pi(s)=\text{rev}\,\pi(s')$ where $\pi=\pi_{\{x\},y}$.
\end{Ex}

The previous example shows that we cannot always reduce to the case where $X$ has two elements.
\begin{Def}
	Let $s$, $s':I_m\to X$ where $X$ is any set. An element $A\in\text{Im}\,s$ is called:
	\begin{enumerate}
		\renewcommand{\labelenumi}{(\roman{enumi})}
		\item direct for the pair $(s,s')$ if for all $i\in s^{-1}(A)$, $s'(i)=A$,
		\item reverse for the pair $(s,s')$ if for all $i\in s^{-1}(A)$, $\text{rev}\,s'(i)=A$.
	\end{enumerate}
\end{Def}
\begin{Ex}
	Let $s, s':I_m\to X$.
	\begin{enumerate}
		\item If each $A\in\text{Im}\,s$ is direct for $(s,s')$ then $s=s'$.
		\item If each $A\in\text{Im}\,s$ is reverse for $(s,s')$ then $s=\text{rev}\,s'$.
		\item The converses of (1) and (2) hold. That is, if $s=s'$ ($s=\text{rev}\,s'$, respectively), then each $A\in\text{Im}\,s$ is direct (reverse, respectively) for $(s,s')$.
	\end{enumerate}
\end{Ex}
The motivation for the previous definition is as follows.
\begin{Ex}
	Let $s, s':I_m\to X$. If there exists an $A\in\text{Im}\,s$ that is neither direct nor reverse for $(s,s')$ then $s$ and $s'$ are not mirrored. 
	
In this case we can make a reduction: letting $\pi=\pi_{\{A\},B}$, then $s_0=\pi(s)$ and $s_0'=\pi(s')$, we have $s_0\ne s_0'$. The last statement holds since there is an $i$ such that $s(i)=A\ne s'(i)$, hence $\pi(s)(i)=A\ne B=\pi(s')(i)$. By a similar argument, $s_0\ne\text{rev}\,s_0'$. In particular, $s_0$ and $s_0'$ are not mirrored by Proposition \ref{ME_twoelements}, and this implies that $s$ and $s'$ are not mirrored.
\end{Ex}

Now, we focus on the case where the reduction to the case of two elements is not possible.
\begin{Def}
	Let $s, s':I_m\to X$ where $X$ is any set. We say that $(s,s')$ is a special pair if every $A\in\text{Im}\,s$ is either direct or reverse for $(s,s')$.
\end{Def}
It is easy to construct examples of special pairs. A particular consequence of the previous examples is the following:
\begin{Lemma}\label{mirror_special}
	If $s, s':I_m\to X$ are mirrored then $(s,s')$ is a special pair.
\end{Lemma}
\begin{Ex}
	Let $(s,s')$ be a special pair, and assume that $A=s(1)$ is direct for $(s,s')$. Then there is a coincidence in $(s,s')$.
\end{Ex}
\begin{Ex}
	Let $(s,s')$ be a special pair. For every $i=1$, 2, \dots, $m$, let $t_i=(1\quad\tau_m(i))=(i\quad m-i+1)$ be a transposition. Let $A=s(i)$, if $s(i)\ne s'(i)$ then $A$ is not direct, and necessarily $A$ is reverse hence $s(i)=t_is'(i)=s'(m-i+1)$. Let $B=s(m-i+1)$. There are two possibilities: $B$ is direct, which implies $A=B$ or $B$ is reverse, and so $s'(i)=B$. In particular, $(s,s')$ is a special pair if and only if $(s',s)$ is a special pair.
\end{Ex}
\begin{Ex}
	Let $(s,s')$ be a special pair. Then $(\text{rev}\,s,s')$ and $(s,\text{rev}\,s')$ are special pairs. Moreover, $A$ is direct (reverse, respectively) for $(s,s')$ if and only if $A$ is reverse (direct, respectively) for $(s,\text{rev}\,s')$. Analogously for $(\text{rev}\,s,s')$.
\end{Ex}

\begin{Lemma}
	Let $s:I_m\to X$ be such that $(s,s')$ is a special pair where $s'=\tau_{m-1}s$. Then $s=s'$.
\end{Lemma}
\begin{proof}
(Sketch) Let $A=s(1)$. If $A$ is direct for $(s,s')$ then $s(1)=s'(1)=A$, and $s(m-1)=\tau_{m-1}s(1)=s'(1)=A$. Hence, since $A$ is direct, $s'(m-1)=A$ and we can proceed the argument inductively.

If $A$ is reverse, then $A=s'(m)=s(m)$ hence $s'(1)=A$. So $s(m-1)=\tau_{m-1}s(1)=s'(1)=A$, hence $s'(2)=A$, which implies $s(m-2)=A$. Continuing the process, the lemma is proved.
\end{proof}
Note that a similar argument can be used to prove the following:
\begin{Lemma}\label{end_lemma}
	Let $(s,s')$ be a special pair, assume $(m_1,m_2)$ a coincidence of $(s,s')$ and let $s_0=s\mid_{I_m'},s_0'=s'\mid_{I_m'}$. If $s_0=\text{rev}'\,s_0'$ then $s=s'$.
\end{Lemma}
We are in a position to prove the main theorem.
\begin{proof}[Proof of Theorem \ref{mainthm}]
	Assume $s$ and $s'$ mirrored. Then $(s,s')$ is a special pair, by Lemma \ref{mirror_special}, and we can change $s'$ to $\text{rev}\,s'$ if necessary, to guarantee that there is a coincidence in $(s,s')$.

	Now $s_0=s\mid_{I_m'}$ and $s_0'=\mid_{I_m'}$ are mirrored, by Lemma \ref{nice_lemma}, and by the induction hypothesis, $s_0=s_0'$ or $s_0=\text{rev}'\,s_0'$. Both cases imply $s=s'$, by Lemma \ref{end_lemma}, proving the theorem.
\end{proof}

\section*{Acknowledgment.}
	The authors are grateful to Prof. Plamen Koshlukov for reading the manuscript, and for the cordial and useful discussions.

\end{document}